\documentclass[a4paper,12pt]{amsart}
\usepackage{amssymb,amsmath,amsfonts, amsthm}
\usepackage{hyperref}
\usepackage{dsfont}
\usepackage{upgreek}
\usepackage{mathrsfs}
\usepackage{mathabx,yfonts}

\newtheorem{thm}{Theorem}

\newtheorem{lem}{Lemma}

\newtheorem*{ack}{Acknowledgements}

\numberwithin{equation}{section}

\begin{document}

\title{On the saturation number for cubic surfaces}

\author{Yuchao Wang}

\address{
School of Mathematics, Shandong University, Jinan, 250100, China}

\email{yuchaowang@mail.sdu.edu.cn}

\begin{abstract}
We investigate the density of rational points on the Fermat cubic surface and the Cayley cubic surface whose coordinates have few prime factors. The key tools used are the weighted sieve, the circle method and universal torsors.
\end{abstract}


\maketitle

\section{Introduction}
In this paper we are concerned with the almost prime integral points on cubic surfaces.

Let $P_r$ indicate an $r$-almost prime, which is a number with at most $r$ prime factors, counted with multiplicity. Furthermore, let $\mathbb{Z}^4_{\text{prim}}$ be the set of vectors $\mathbf{x}=(x_0,x_1,x_2,x_3)\in\mathbb{Z}^4$ with $\text{gcd}(x_0,x_1,x_2,x_3)=1$. For any cubic surface $S\subset\mathbb{P}^3$ defined over $\mathbb{Q}$, we define the saturation number $r(S)$ to be the least number $r$ such that the set of $\mathbf{x}\in \mathbb{Z}^4_{\text{prim}}$ for which $[\mathbf{x}]\in S$ and $x_0x_1x_2x_3=P_r$ , is Zariski dense in $S$. The main aim of this paper is to show that $r(S)<\infty$ for a non-singular surface (the Fermat cubic surface) and a singular surface (the Cayley cubic surface).

Bourgain, Gamburd and Sarnak \cite{BGS} and Nevo and Sarnak \cite{NS} established upper bounds for saturation numbers for orbits of congruence subgroups of semi-simple groups acting linearly on affine space. Moreover, Liu and Sarnak \cite{LS} considered the saturation number for certain affine quadric surfaces. However, these results do not cover the surfaces considered here.

The Fermat cubic surface is defined in $\mathbb{P}^3$ by the equation
\begin{equation*}
S_1:\ x_0^3+x_1^3+x_2^3+x_3^3=0.
\end{equation*}
It is non-singular and has 27 lines, three of which are rational and take the form $x_i+x_j=x_k+x_l=0$.

The Cayley cubic surface is defined in $\mathbb{P}^3$ by the equation
\begin{equation*}
S_2:\ x_1x_2x_3+x_0x_2x_3+x_0x_1x_3+x_0x_1x_2=0.
\end{equation*}
It has singularity type $4\mathbf{A}_1$. Moreover, there are 9 lines in the surface, three of which have the form $x_i+x_j=x_k+x_l=0$, and the remaining six have the shape $x_i=x_j=0$.

When considering the almost prime points on the surface, we are looking for the set of such points which form a Zariski dense subset. Thus it is not enough to work with almost prime points lying on individual curves contained in the surface. In particular, by writing $U$ for the complement of the lines in the surface $S$, we may restrict attention to the open subset $U\subset S$.

To prove that a set of $[\mathbf{x}]\in S$ is Zariski dense, it suffices to prove that given $\varepsilon>0$ and any $\boldsymbol{\xi}\in\mathbb{R}^4$ satisfying $[\boldsymbol{\xi}]\in U$, there exists $B\in \mathbb{N}$ sufficiently large and at least one point $[\mathbf{x}]$ in the set, such that
\begin{equation*}
\Big|\frac{\mathbf{x}}{B}-\boldsymbol{\xi}\Big|< \varepsilon.
\end{equation*}

The following are our main results and establish the finiteness of $r(S_1)$ and $r(S_2)$.

\begin{thm}
Let $[\boldsymbol{\xi}]\in U_1(\mathbb{R})$ and $\varepsilon>0$. Define
\begin{equation*}
M_{U_1}(\boldsymbol{\xi},\varepsilon,B,r)=\# \left\{\mathbf{x}\in\mathbb{Z}^4_{\text{\em{prim}}}:
\begin{aligned}
&[\mathbf{x}]\in U_1,\,\Big|\frac{\mathbf{x}}{B}-\boldsymbol{\xi}\Big|<\varepsilon,\\
&x_0x_1x_2x_3=P_r
\end{aligned}
\right\}.
\end{equation*}
Then for sufficiently large $B$, we have
\begin{equation*}
M_{U_1}(\boldsymbol{\xi},\varepsilon,B,20)\gg B(\log B)^{-4}.
\end{equation*}
The implicit constant is allowed to depend on $\boldsymbol{\xi}$ and $\varepsilon$. In particular, the saturation number satisfies $r(S_1)\leq 20$.
\end{thm}
\begin{thm}
Let $[\boldsymbol{\xi}]\in U_2(\mathbb{R})$. Define
\begin{equation*}
M_{U_2}(\boldsymbol{\xi},B,r)=\# \left\{\mathbf{x}\in\mathbb{Z}^4_{\text{\em{prim}}}:
\begin{aligned}
&[\mathbf{x}]\in U_2,\,\Big|\frac{\mathbf{x}}{B}-\boldsymbol{\xi}\Big|\ll (\log B)^{-1},\\
&x_0x_1x_2x_3=P_r
\end{aligned}
\right\},
\end{equation*}
then for sufficiently large $B$, we have
\begin{equation*}
M_{U_2}(\boldsymbol{\xi},B,12)\gg B(\log B)^{-7}.
\end{equation*}
The implicit constants are allowed to depend on $\boldsymbol{\xi}$. In particular, the saturation number satisfies $r(S_2)\leq 12$.
\end{thm}

These results show that any real point on $U_i$ can be approximated arbitrarily closely by \emph{multiplicatively constrained} rational points on $U_i$, for $i=1,2$.

It is of interest to compare our result with the density of rational points on the cubic surfaces. We set
\begin{equation*}
N_{U}(B)=\#\{\mathbf{x}\in \mathbb{Z}^4_{\text{prim}}:\,[\mathbf{x}]\in U,\,\max |x_i|\leq B \}.
\end{equation*}
Manin (see Batyrev and Manin \cite{BM}) has given a very general conjecture which would predict that there is a suitable positive constant $c$ such that
\begin{equation*}
N_{U}(B)\sim cB(\log B)^{\rho-1},
\end{equation*}
as $B\rightarrow \infty$, where $\rho$ is the rank of the Picard group of the surface. The Manin conjecture is still open for $S_1$ and $S_2$.

For the Fermat cubic surface, we have
\begin{equation*}
B(\log B)^3\ll N_{U_1}(B)\ll B^{\frac{4}{3}+\varepsilon},
\end{equation*}
for any $\varepsilon>0$. The upper bound is established by Heath-Brown \cite{Hb4} and the lower bound is due to Sofos \cite{So}. The order of magnitude of this lower bound agrees with the Manin conjecture.

For the Cayley cubic surface, upper and lower bounds of the expected order of magnitude have been established by Heath-Brown \cite{Hb}, i.e.
\begin{equation*}
B(\log B)^6\ll N_{U_2}(B)\ll B(\log B)^6.
\end{equation*}
The approach combines analytic methods with the theory of universal torsors. For details of universal torsors for singular del Pezzo surfaces, we refer the reader to Derenthal \cite{D}.

Our proof of Theorems 1 and 2 relies on convenient parameterisations of points on the surfaces. For the Fermat cubic surface, we use Euler's parametrisation to get integral points on the surface, for which each of the coordinates is a ternary cubic form. We then apply a weighted sieve to get almost prime points on the surface. For the Cayley cubic surface, we apply the theory of universal torsors to specify some integral solutions in a particular form, whose coordinates have few prime factors, and then give a lower bound for the number of such solutions which are close to some fixed real solution via the circle method. More precisely, we use the circle method to count the number of solutions to the equation
\begin{equation*}
\beta_0p_0+\beta_1p_1+\beta_2p_2+\beta_3p_3=0,
\end{equation*}
where $p_j$ are primes which lie in certain intervals and $\beta_j\in\{1,-1\}$, for $j=0,\dots,3$.

It seems likely that similar methods to those developed in this paper will apply to other cubic surfaces. In the setting of singular surfaces, for example, let $S_3\subset \mathbb{P}^3$ be the cubic surface given by the equation
\begin{equation*}
x_0x_1x_2=x_3(x_0+x_1+x_2)^2.
\end{equation*}
There is a unique singular point which is of type $\mathbf{D}_4$. The density of rational points of bounded height on $S_3$ has been studied by Browning \cite{Br} and Le Boudec \cite{Le}. Arguing much as in this paper, one can show that $r(S_3)\leq 12$.

For a cubic surface $S$, we define $\tilde{r}(S)$ to be the least number $\tilde{r}$ such that the set of $\mathbf{x}\in \mathbb{Z}^4_{\text{prim}}$ for which $[\mathbf{x}]\in S$ and the product $x_0x_1x_2x_3$ has at most $\tilde{r}$ \emph{distinct} prime factors, is Zariski dense in $S$. Then it is worth pointing out that for the Cayley cubic surface, our methods give $\tilde{r}(S_2)\leq 4$. However, for the Fermat cubic surface, we only get $\tilde{r}(S_1)\leq 20$, since we use sieve methods instead of the circle method.

It is natural to consider how close our upper bound is to the truth. For $S_1$, we conjecture $r(S_1)=4$. Considering $S_2$, from the parametrisation of the points presented in Section 3.3, we see that one cannot expect the set of the points on the Cayley cubic surface, for which the product of the coordinates has at most $5$ prime factors, to be Zariski dense. Moreover, there exists a point $(3,5,7,105)$ on the Cayley cubic surface, for which the product of the coordinates has $6$ prime factors. Consequently, the saturation number for the Cayley cubic surface satisfies $r(S_2)\geq 6$.

Throughout the paper, we let the letter $p$, with or without indices, be reserved for prime numbers. Let $\varepsilon$ denote a  small positive constant, not necessarily the same in all occurrences. As usual, let $\mu(n)$, $\varphi(n)$ and $\tau(n)$ denote the M\"{o}bius function, Euler's totient function and the divisor function respectively. We also write $e(\alpha)=e^{2\pi i \alpha}$, $e_d(\alpha)=e^{\frac{2\pi i \alpha}{d}}$ and $(a,b)=\text{gcd}(a,b)$. Finally, we use $m\sim M$ as the abbreviation for the condition $M\leq m<2M$.

\begin{ack}
\emph{The author wishes to express his sincere appreciation to Tim Browning for introducing him to this problem and giving him various suggestions. This work was carried out while the author was a visiting PhD student at University of Bristol. The author is grateful for the hospitality. Thanks also to Pierre Le Boudec for many useful conversations. While working on this paper the author was supported by the China Scholarship Council.}
\end{ack}

\section{The Fermat cubic surface}
\subsection{Preliminary steps}
Considering the Fermat cubic surface , we use the fact that $S_1$ is rational. Thus we have the following map
\begin{equation*}
\begin{split}
\mathbb{P}^2&\rightarrow S_1\\
[\mathbf{y}]&\mapsto [\mathbf{F}(\mathbf{y})],
\end{split}
\end{equation*}
where $\mathbf{y}=(y_1,y_2,y_3)$ and $\mathbf{F}(\mathbf{y})=(F_0(\mathbf{y}),F_1(\mathbf{y}),F_2(\mathbf{y}),F_3(\mathbf{y}))$. We get the parametrisation of the points on the Fermat cubic surface by using Euler's parametrisation.  Let
\begin{equation}
\begin{split}
&H_0(\mathbf{y})=-6y_1y_2y_3,\\
&H_1(\mathbf{y})=y_1(y_1^2+3y_2^2+3y_3^2),\\
&H_2(\mathbf{y})=y_2(y_1^2+3y_2^2+9y_3^2),\\
&H_3(\mathbf{y})=3y_3(y_1^2+y_2^2+3y_3^2).
\end{split}
\end{equation}
Then we may take
\begin{equation}
\begin{split}
&F_0(\mathbf{y})=H_0(\mathbf{y})+H_1(\mathbf{y})+H_2(\mathbf{y})+H_3(\mathbf{y}),\\
&F_1(\mathbf{y})=H_0(\mathbf{y})+H_1(\mathbf{y})-H_2(\mathbf{y})-H_3(\mathbf{y}),\\
&F_2(\mathbf{y})=H_0(\mathbf{y})-H_1(\mathbf{y})+H_2(\mathbf{y})-H_3(\mathbf{y}),\\
&F_3(\mathbf{y})=H_0(\mathbf{y})-H_1(\mathbf{y})-H_2(\mathbf{y})+H_3(\mathbf{y}).
\end{split}
\end{equation}
By arguments in Section 11.9 of Hua \cite{Hua}, we see that
\begin{equation*}
F_0^3(\mathbf{y})+F_1^3(\mathbf{y})+F_2^3(\mathbf{y})+F_3^3(\mathbf{y})=0.
\end{equation*}
Furthermore, we define
\begin{equation*}
F(\mathbf{y})=F_0(\mathbf{y})F_1(\mathbf{y})F_2(\mathbf{y})F_3(\mathbf{y}).
\end{equation*}
\begin{lem}
For $i=0,\dots,3$, each of $F_i(\mathbf{y})$ is non-singular. Consequently, there exists a constant $C_1$ which only depends on $F_i(\mathbf{y})$, such that for $p\geq C_1$, each of $F_i(\mathbf{y})$ is non-singular over $\mathbb{F}_p$.
\end{lem}
\begin{proof}
It suffices to show that each of $F_i(\mathbf{y})$ is non-singular. We have
\begin{equation*}
\begin{split}
&\frac{\partial F_0}{\partial y_1}=3y_1^2+3y_2^2+3y_3^2+2y_1y_2+6y_1y_3-6y_2y_3,\\
&\frac{\partial F_0}{\partial y_2}=y_1^2+9y_2^2+9y_3^2+6y_1y_2-6y_1y_3+6y_2y_3,\\
&\frac{\partial F_0}{\partial y_3}=3y_1^2+3y_2^2+27y_3^2-6y_1y_2+6y_1y_3+18y_2y_3.
\end{split}
\end{equation*}
Thus $(\frac{\partial F_0}{\partial y_1},\frac{\partial F_0}{\partial y_2},\frac{\partial F_0}{\partial y_3})=(0,0,0)$ is equivalent to
\begin{equation*}
\begin{split}
&y_1^2=3y_2y_3-3y_1y_3,\\
&y_2^2=y_1y_3-y_1y_2,\\
&3y_3^2=y_1y_2-3y_2y_3.
\end{split}
\end{equation*}
If $y_1=0$, then we have $y_2=0$ and $y_3=0$. Furthermore, if $y_2=y_3$, then this yields $y_2=0$ and we only have trivial solutions. Now we assume that none of $y_1,y_2,y_3$ is zero and $y_2\neq y_3$. Then we get
\begin{equation*}
y_1=\frac{y_2^2}{y_3-y_2}.
\end{equation*}
Thus we obtain
\begin{equation*}
\begin{split}
&y_2^3+3y_2^2y_3-3y_3^3=0,\\
&y_2^3-6y_2^2y_3+9y_2y_3^2-3y_3^3=0.
\end{split}
\end{equation*}
The system of equations only has trivial solution, then we obtain that $F_0(\mathbf{y})$ is non-singular. Arguing similarly, we see that each of $F_i(\mathbf{y})$ is non-singular.
\end{proof}
\begin{lem}
There exists a constant $C_2$, which only depends on $F(\mathbf{y})$, such that we have
\begin{equation*}
|\#\{\mathbf{y}(\text{\em{mod}}\,p):\ p|F(\mathbf{y})\}- 4p^2|\leq C_2 p^{\frac{3}{2}},
\end{equation*}
provided that $p\geq C_2$.
\end{lem}
\begin{proof}
By Lemma 1, we see that there exists a constant $C_1$, such that for $p\geq C_1$, each of $F_i(\mathbf{y})$ is non-singular over $\mathbb{F}_p$. It follows from Theorem 8.1 of Deligne \cite{Deligne} that there exists a constant $A_1$, which only depends on $F_i(\mathbf{y})$, such that we get
\begin{equation*}
|\#\{\mathbf{y}(\text{mod}\,p):\ p|F_i(\mathbf{y})\}- p^2|\leq A_1 p^{\frac{3}{2}}.
\end{equation*}
Moreover, we have
\begin{equation*}
\#\{\mathbf{y}(\text{mod}\,p):\ p|F_i(\mathbf{y}),\,p|F_j(\mathbf{y}),\  \text{for}\ i\neq j\}\leq A_2 p,
\end{equation*}
for some constant $A_2$ which only depends on $F_i(\mathbf{y})$. Thus we prove the lemma.
\end{proof}
\subsection{Level of distribution}
In this section, we prove a level of distribution formula for ternary cubic forms. Marasingha \cite{Mar} used the geometry of numbers to prove a level of distribution formula for binary forms. However, his techniques do not extend to ternary forms.

By Lemma 1 and Lemma 2, we see that there exists a constant $C\geq5$, which only depends on $F_i(\mathbf{y})$, such that each of $F_i(\mathbf{y})$ is non-singular over $\mathbb{F}_p$ and
\begin{equation*}
|\#\{\mathbf{y}(\text{mod}\,p):\ p|F(\mathbf{y})\}- 4p^2|\leq C p^{\frac{3}{2}},
\end{equation*}
provided that $p\geq C$. Define
\begin{equation*}
D=\text{rad}\Big(\text{Res}(F_1,F_2,F_3)\prod_{p\leq C}p \Big),
\end{equation*}
where $\text{rad}(n)$ is the squarefree kernel of $n$ and $\text{Res}(F_1,F_2,F_3)$ is the resultant of $F_1$, $F_2$ and $F_3$. Let $\mathbf{z}=(1,0,0)$. Then for $\mathbf{y}\equiv \mathbf{z}\,(\text{mod}\,D)$, we have $(F_i(\mathbf{y}),D)=1$. Write
\begin{equation*}
\begin{split}
\Psi&=\{\mathbf{y}\in \mathbb{Z}^3:\ \mathbf{y}\equiv \mathbf{z}\,(\text{mod}\,D)\},\\
\Lambda_{d}&=\{ \mathbf{y}\in \mathbb{Z}^3:\,d|F(\mathbf{y})\},\\
\rho(d)&=\#\{\mathbf{y}(\text{mod}\,d):\ d|F(\mathbf{y})\}.
\end{split}
\end{equation*}
We prove the following level of distribution formula.
\begin{lem}
Let
\begin{equation*}
L(M,Q)=\sum_{\stackrel{d\leq Q}{(d,D)=1}}\sup_{\partial(\mathcal {R})\leq M}\mu^2(d)\Big| \#\Lambda_{d}\cap\mathcal {R}\cap\Psi-\frac{\rho(d)\text{\emph{vol}}(\mathcal{R})}{d^3D^3}\Big|,
\end{equation*}
where $\mathcal{R}$ is a convex subset of $\mathbb{R}^3$ with piecewise continuously differentiable boundary. Then for any $\varepsilon>0$, we have
\begin{equation*}
L(M,Q)\ll (M^2Q^{\frac{1}{2}}+MQ+Q^2)Q^{\varepsilon}.
\end{equation*}
\end{lem}
\begin{proof}
Note that
\begin{equation*}
\#\Lambda_{d}\cap\mathcal {R}\cap\Psi=\sum_{\stackrel{\mathbf{y}\in\mathcal {R}\cap\Psi}{d|F(\mathbf{y})}}1=\sum_{\stackrel{\mathbf{u}(\text{mod}\,d)}{d|F(\mathbf{u})}} \sum_{\stackrel{\mathbf{y}\in\mathcal {R}\cap\Psi}{\mathbf{y}\equiv\mathbf{u}(\text{mod}\,d) }}1.
\end{equation*}
Thus we have
\begin{equation*}
\begin{split}
\#\Lambda_{d}\cap\mathcal {R}\cap\Psi
=&\frac{1}{d^3}\sum_{\mathbf{m}(\text{mod}\,d)}\sum_{\stackrel{\mathbf{u}(\text{mod}\,d)}{d|F(\mathbf{u})}}e_d(-\mathbf{m}\cdot\mathbf{u}) \sum_{{\mathbf{y}\in\mathcal {R}\cap\Psi}}e_d(\mathbf{m}\cdot\mathbf{y})\\
=&\frac{1}{d^3}\sum_{\stackrel{\mathbf{m}(\text{mod}\,d)}{\mathbf{m}\neq \mathbf{0}}}\sum_{\stackrel{\mathbf{u}(\text{mod}\,d)}{d|F(\mathbf{u})}}e_d(-\mathbf{m}\cdot\mathbf{u}) \sum_{{\mathbf{y}\in\mathcal {R}\cap\Psi}}e_d(\mathbf{m}\cdot\mathbf{y})\\
&+\frac{1}{d^3}\rho(d)\sum_{{\mathbf{y}\in\mathcal {R}\cap\Psi}}1.
\end{split}
\end{equation*}
Note that
\begin{equation*}
\sum_{{\mathbf{y}\in\mathcal {R}\cap\Psi}}1=\frac{\text{vol}(\mathcal{R})}{D^3}+O(M^2).
\end{equation*}
Then we obtain
\begin{equation}
\begin{split}
 &\#\Lambda_{d}\cap\mathcal {R}\cap\Psi-\frac{\rho(d)\text{vol}(\mathcal{R})}{d^3D^3}\\
&=\frac{1}{d^3}\sum_{\stackrel{\mathbf{m}(\text{mod}\,d)}{\mathbf{m}\neq \mathbf{0}}}\sum_{\stackrel{\mathbf{u}(\text{mod}\,d)}{d|F(\mathbf{u})}}e_d(-\mathbf{m}\cdot\mathbf{u}) \sum_{{\mathbf{y}\in\mathcal {R}\cap\Psi}}e_d(\mathbf{m}\cdot\mathbf{y})+O(M^2).
\end{split}
\end{equation}
Write
\begin{equation*}
S_p(\mathbf{m})=\sum_{\stackrel{\mathbf{u}(\text{mod}\,p)}{p|F(\mathbf{u})}}e_p(-\mathbf{m}\cdot\mathbf{u}).
\end{equation*}
Then we get
\begin{equation*}
\sum_{\stackrel{\mathbf{u}(\text{mod}\,d)}{d|F(\mathbf{u})}}e_d(-\mathbf{m}\cdot\mathbf{u})=\prod_{p|d} S_p(\mathbf{m}).
\end{equation*}
Using the inclusion-exclusion principle, we have
\begin{equation*}
S_p(\mathbf{m})-\sum_{i=0}^3\sum_{\stackrel{\mathbf{y}(\text{mod}\,p)}{p|F_i(\mathbf{y})}}e_{p}(\mathbf{m}\cdot\mathbf{y})\ll p.
\end{equation*}
By Lemma 6 of  Heath-Brown \cite{Hb2}, we see that for each $F_i(\mathbf{y})$, there exists the dual form $G_i(\mathbf{m})$ which is a ternary form. Moreover, the degree of each $G_i(\mathbf{m})$ is 6. Arguing similarly as in Section 2 of Heath-Brown \cite{Hb3}, we obtain
\begin{equation*}
\sum_{\stackrel{\mathbf{y}(\text{mod}\,p)}{p|F_i(\mathbf{y})}}e_{p}(\mathbf{m}\cdot\mathbf{y})\ll p (p,G_i(\mathbf{m}))^{\frac{1}{2}},
\end{equation*}
where the implicit constant only depends on $F_i(\mathbf{y})$. Write
\begin{equation*}
G(\mathbf{m})=\prod_{i=0}^3G_i(\mathbf{m}).
\end{equation*}
Then we get
\begin{equation*}
|S_p(\mathbf{m})|\leq Ap (p,G(\mathbf{m}))^{\frac{1}{2}},
\end{equation*}
where $A$ is a constant only depending on $F_i(\mathbf{y})$. It follows that
\begin{equation}
\Big|\sum_{\stackrel{\mathbf{u}(\text{mod}\,d)}{d|F(\mathbf{u})}}e_d(-\mathbf{m}\cdot\mathbf{u})\Big|\leq d^{\varepsilon} (d,G(\mathbf{m}))^{\frac{1}{2}}.
\end{equation}
Furthermore, for $\partial(\mathcal {R})\leq M$, we have
\begin{equation}
\sum_{{\mathbf{y}\in\mathcal {R}\cap\Psi}}e_d(\mathbf{m}\cdot\mathbf{y})\ll\prod_{j=1}^3\min\Big\{M,\frac{1}{\|\frac{m_jD}{d}\|}\Big\}.
\end{equation}
Combining (2.3) with (2.4) and (2.5), we get
\begin{equation}
\begin{split}
L(M,Q)&\ll \sum_{d\leq Q}\frac{1}{d^3}\sum_{\stackrel{\mathbf{m}(\text{mod}\,d)}{\mathbf{m}\neq \mathbf{0}}}d^{1+\varepsilon} (d,G(\mathbf{m}))^{\frac{1}{2}}\prod_{j=1}^3\min\big\{M,\frac{d}{m_j}\big\}+O(M^2)\\
&\ll\sum_{\stackrel{0\leq m_j\leq Q}{\mathbf{m}\neq \mathbf{0}}}\sum_{d\leq Q}d^{-2+\varepsilon} (d,G(\mathbf{m}))^{\frac{1}{2}}\prod_{j=1}^3\min\big\{M,\frac{d}{m_j}\big\}+O(M^2).
\end{split}
\end{equation}
Let $T_1$ and $T_2$ denote the contribution from $\mathbf{m}$ with $G(\mathbf{m})\neq0$ and $G(\mathbf{m})=0$, respectively. Using the bound $\tau(n)\ll n^{\varepsilon}$, we have
\begin{equation}
\begin{split}
T_1\ll&\sum_{\stackrel{0\leq m_j\leq Q}{G(\mathbf{m})\neq0}}\sum_{h|G(\mathbf{m})}\sum_{d'\leq \frac{Q}{h}}(hd')^{-2+\varepsilon}h^{\frac{1}{2}}
\prod_{j=1}^3\min\big\{M,\frac{hd'}{m_j}\big\}\\
&\ll M^2Q^{\varepsilon}+MQ^{1+\varepsilon}+Q^{2+\varepsilon}.
\end{split}
\end{equation}
Now we proceed to estimate $T_2$. Note that if $F_i(\mathbf{y})=\mathbf{m}\cdot\mathbf{y}=0$ defines a non-singular variety, then we get $G_i(\mathbf{m})\neq0$. Moreover, $F_0(\mathbf{y})=y_1=0$ defines a non-singular variety, thus we obtain $G_0(1,0,0)\neq0$. By similar arguments, we see that $m_j^6$ appears with non-zero coefficient in each $G_i(\mathbf{m})$. If $\mathbf{m}$ satisfies $m_1m_2m_3=0$, $G(\mathbf{m})=0$ and $\mathbf{m}\neq \mathbf{0}$, then we obtain that only one of $m_j$ vanishes. Without loss of generality, we may assume $m_1=0$. Write
\begin{equation*}
T'_2=\sum_{\stackrel{1\leq m_j\leq Q}{G(\mathbf{m})=0}}\sum_{d\leq Q}d^{-2+\varepsilon} (d,G(\mathbf{m}))^{\frac{1}{2}}\prod_{j=1}^3\min\big\{M,\frac{d}{m_j}\big\}.
\end{equation*}
 Then we get
\begin{equation}
\begin{split}
T_2-T'_2&\ll M^2\sum_{1\leq m_2\leq Q}m_2^{-1}\sum_{\stackrel{1\leq m_3\leq Q}{G(0,m_2,m_3)=0}} \sum_{d\leq Q}d^{-\frac{1}{2}+\varepsilon}\\
&\ll M^2\sum_{1\leq m_2\leq Q}m_2^{-1}\sum_{d\leq Q}d^{-\frac{1}{2}+\varepsilon}\\
&\ll M^2Q^{\frac{1}{2}+\varepsilon}.
\end{split}
\end{equation}
It follows from Theorem 3 of Heath-Brown \cite{Hb5} that
\begin{equation*}
\#\{\mathbf{m}\in \mathbb{Z}^3_{\text{prim}}:\,G_i(\mathbf{m})=0, |m_j|\leq M_j\}\ll (M_1M_2M_3)^\frac{1}{9},
\end{equation*}
where $M_1$, $M_2$ and $M_3$ are positive parameters and $\mathbb{Z}^3_{\text{prim}}$ is the set of vectors $\mathbf{m}=(m_1,m_2,m_3)\in\mathbb{Z}^3$ with $\text{gcd}(m_1,m_2,m_3)=1$. Thus we obtain
\begin{equation*}
\begin{split}
&\#\{\mathbf{m}\in \mathbb{Z}^3:\,G_i(\mathbf{m})=0, |m_j|\leq M_j\}\\
&\ll\sum_{h\leq(M_1M_2M_3)^{\frac{1}{3}}}\#\Big\{\mathbf{m}\in \mathbb{Z}^3_{\text{prim}}:\,G_i(\mathbf{m})=0, |m_j|\leq \frac{M_j}{h}\Big\}\\
&\ll (M_1M_2M_3)^{\frac{1}{3}}.
\end{split}
\end{equation*}
Then a familiar dyadic dissection argument yields that there exist numbers
\begin{equation*}
0<M'_1,M'_2,M'_3\leq Q,
\end{equation*}
such that we have
\begin{equation*}
T'_2\ll Q^{\varepsilon}\sum_{\stackrel{m_j\sim M'_j}{G(\mathbf{m})=0}}\sum_{d\leq Q}d^{-\frac{3}{2}+\varepsilon}\prod_{j=1}^3\min\big\{M,\frac{d}{m_j}\big\}.
\end{equation*}
Using the bound
\begin{equation*}
\min\big\{M,\frac{d}{m_j}\big\}\leq M^{\frac{2}{3}}\Big(\frac{d}{m_j}\Big)^{\frac{1}{3}},
\end{equation*}
we get
\begin{equation}
\begin{split}
T'_2&\ll M^2Q^{\varepsilon}(M'_1M'_2M'_3)^{-\frac{1}{3}}\sum_{\stackrel{m_j\sim M'_j}{G(\mathbf{m})=0}}\sum_{d\leq Q}d^{-\frac{1}{2}+\varepsilon}\\
&\ll M^2Q^{\frac{1}{2}+\varepsilon}.
\end{split}
\end{equation}
Combining (2.8) with (2.9), we get
\begin{equation}
T_2\ll MQ^{\frac{3}{2}+\varepsilon}.
\end{equation}
Inserting (2.7) and (2.10) into (2.6), we obtain
\begin{equation*}
L(M,Q)\ll M^2Q^{\frac{1}{2}+\varepsilon}+MQ^{1+\varepsilon}+Q^{2+\varepsilon}.
\end{equation*}
Thus the lemma is proved.
\end{proof}
\subsection{The weighted sieve}
In this section, we get almost prime value of $F(\mathbf{y})$. We use the sieve of Diamond and Halberstam to prove the following lemma.
\begin{lem}
Let
\begin{equation*}
\mathfrak{A}=\{F(\mathbf{y}):\ \mathbf{y}\in B^{\frac{1}{3}}\mathcal{R}^{(0)}\cap \Psi\},
\end{equation*}
where $B$ is a large parameter and $\mathcal{R}^{(0)}$ is some fixed convex subset of $\mathbb{R}^3$ with piecewise continuously differentiable boundary. Then we have
\begin{equation*}
\#\{P_r:\ P_r\in\mathfrak{A}\}\gg B(\log B)^{-4},
\end{equation*}
provided that $r\geq20$.
\end{lem}
\begin{proof}
Define the set $\mathfrak{P}$ of primes which do not divide $D$, and $\bar{\mathfrak{P}}$ to be the complement of $\mathfrak{P}$ in the set of all primes. Let
\begin{equation*}
Y=\frac{\text{vol}(\mathcal{R}^{(0)})B}{D^3}.
\end{equation*}
Recall that $\rho(p)=\#\{\mathbf{y}(\text{mod}\,p):\ p|F(\mathbf{y})\}$. Introduce the multiplicative function $\omega(\cdot)$ which satisfies $\omega(1)=1$, $\omega(p)=0$ for $p\in\bar{\mathfrak{P}}$ and for $p\in\mathfrak{P}$,
\begin{equation*}
\omega(p)=\frac{\rho(p)}{p^2}.
\end{equation*}
Write
\begin{equation*}
R_d=\#\{a\in \mathfrak{A}:\,d|a\}-\frac{\omega(d)}{d}Y\quad \text{if}\ \mu(d)\neq 0.
\end{equation*}
The following lemma is a restatement of Theorem 1 in Diamond and Halberstam \cite{DH}.
\begin{lem}
Suppose that there exist real constants $\kappa>1$, $A_1,A_2\geq2$ and $A_3\geq1$ such that
\begin{equation}
0\leq \omega (p)<p,
\end{equation}
\begin{equation}
\prod_{z_1\leq p<z}\Big(1-\frac{\omega(p)}{p}\Big)^{-1}\leq \Big(\frac{\log z}{\log z_1}\Big)^{\kappa}\Big(1+\frac{A_1}{\log z_1}\Big) \quad \text{for}\ 2\leq z_1<z,
\end{equation}
\begin{equation}
\sum_{\stackrel{d<Y^{\alpha}(\log Y)^{-A_3}}{(d,\bar{\mathfrak{P}})=1}}\mu^2(d)4^{\nu(d)}|R_d|\leq A_2\frac{Y}{\log^{\kappa+1}Y},
\end{equation}
for some $\alpha$ with $0<\alpha\leq 1$; that
\begin{equation}
(a,\bar{\mathfrak{P}})=1\quad \text{for all}\  a\in\mathfrak{A},
\end{equation}
and that
\begin{equation}
|a|\leq Y^{\alpha\mu}\quad \text{for some} \ \mu, \text{and for all} \ a\in\mathfrak{A}.
\end{equation}
Then there exists a real constant $\beta_{\kappa}>2$ such that, for any real numbers $u$ and $v$ satisfying
\begin{equation*}
\alpha^{-1}<u<v,\qquad \beta_{\kappa}<\alpha v,
\end{equation*}
we have
\begin{equation*}
\#\{P_r:\ P_r\in\mathfrak{A}\}\gg Y\prod_{p<Y^{1/v}}\Big(1-\frac{\omega(p)}{p}\Big),
\end{equation*}
where
\begin{equation*}
r>\alpha\mu u-1+\frac{\kappa}{f_{\kappa}(\alpha v)}\int_{1}^{v/u}F_{\kappa}(\alpha v-s)\left(1-\frac{u}{v}s\right)\frac{d s}{s}.
\end{equation*}
\end{lem}
We need to verify the conditions (2.11)-(2.15).
For $p>C$, we have
\begin{equation*}
\omega(p)\leq5.
\end{equation*}
Then we obtain that (2.11) holds. Taking logarithms, we rewrite (2.12) as requiring that
\begin{equation*}
\sum_{z_1\leq p<z}\sum_{i=1}^{\infty}\frac{\omega(p)^i}{ip^i}\leq \kappa\log\log z-\kappa\log\log z_1+\frac{B_1}{\log z_1},
\end{equation*}
for $z_1\geq C$. The main contribution comes from $\sum\limits_{z_1\leq p<z}\frac{\omega(p)}{p}$. We have
\begin{equation*}
\begin{split}
\sum_{z_1\leq p<z}\frac{\omega(p)}{p}&\leq \sum_{z_1\leq p<z}\frac{4}{p}+\sum_{z_1\leq p<z}\frac{C}{p^{\frac{3}{2}}}\\
&\leq 4\log\log z-4\log\log z_1+\frac{B_1}{\log z_1}.
\end{split}
\end{equation*}
For the error term, we have
\begin{equation*}
\sum_{i=2}^{\infty}\sum_{z_1\leq p<z}\frac{\omega(p)^i}{ip^i}\leq \sum_{i=2}^{\infty}\sum_{z_1\leq p<z}\Big(\frac{5}{p}\Big)^i\ll\frac{1}{\log z_1}.
\end{equation*}
Now we obtain that (2.12) holds for $\kappa=4$. Using Lemma 3, we see that
\begin{equation*}
\sum_{\stackrel{d<{Y^{\alpha}}(\log Y)^{-A_3}}{(d,\bar{\mathfrak{P}})=1}}\mu^2(d)4^{\nu(d)}|R_d|\ll Y^{\frac{\alpha}{2}+\frac{2}{3}+\varepsilon}+Y^{\alpha+\frac{1}{3}+\varepsilon}+Y^{2\alpha+\varepsilon}.
\end{equation*}
Now it suffices to have the following
\begin{equation*}
Y^{\frac{\alpha}{2}+\frac{2}{3}+\varepsilon}+Y^{\alpha+\frac{1}{3}+\varepsilon}+Y^{2\alpha+\varepsilon}\ll Y.
\end{equation*}
Thus the condition (2.13) is satisfied, provided that we choose $\alpha<\frac{1}{2}$. The condition (2.14) holds since we restrict $\mathbf{x}\equiv\mathbf{z}\,(\text{mod}\,D)$. For the condition (2.15), we see that any $\mu>\frac{4}{\alpha}$ is admissible.

Arguing similarly as in Section 6 of Liu and Sarnak \cite{LS}, we obtain that one may take any value of $r$ larger than
\begin{equation*}
m(\lambda):=(3+4\log \beta_4)+(8-\frac{4}{\beta_4}+\log\beta_4)\lambda-4\log\lambda-\lambda\log\lambda.
\end{equation*}
Note that
\begin{equation*}
\beta_4=9.0722\dots.
\end{equation*}
by Appendix III in \cite{DHR}. Then we have
\begin{equation*}
\min_{0<\lambda<\beta_4}m(\lambda)=m(0.4147489\dots)=19.7559\dots.
\end{equation*}
Thus we prove the lemma.
\end{proof}
\subsection{The proof of Theorem 1}
In this section, we prove the Zariski density of the almost prime points. It suffices to show that given $\varepsilon>0$ and any $[\boldsymbol{\xi}]\in U_1$, there exists $B\in\mathbb{N}$ sufficiently large and a fixed convex set $\mathcal{R}^{(0)}\subset\mathbb{R}^3$, such that for $\mathbf{y}\in B^{\frac{1}{3}}\mathcal{R}^{(0)}$, we have
\begin{equation}
\Big|\frac{F_i(\mathbf{y})}{B}-\xi_i\Big|< \varepsilon,
\end{equation}
for $i=0,\dots,3$.
For such a given $\boldsymbol{\xi}$, we write
\begin{equation*}
\begin{split}
&\zeta_0=\frac{1}{4}(\xi_0+\xi_1+\xi_2+\xi_3),\\
&\zeta_1=\frac{1}{4}(\xi_0+\xi_1-\xi_2-\xi_3),\\
&\zeta_2=\frac{1}{4}(\xi_0-\xi_1+\xi_2-\xi_3),\\
&\zeta_3=\frac{1}{4}(\xi_0-\xi_1-\xi_2+\xi_3).
\end{split}
\end{equation*}
By (2.2), we see that it suffices to prove
\begin{equation*}
\Big|\frac{H_i(\mathbf{y})}{B}-\zeta_i\Big|<\frac{1}{4} \varepsilon,
\end{equation*}
where all of $H_i(\mathbf{y})$ are defined as in (2.1). Recall that $\boldsymbol{\xi}$ does not lie in any line of $S_1$. Thus we get
\begin{equation*}
\begin{split}
&\zeta_0\zeta_2+3\zeta_1\zeta_3=\frac{1}{4}(\xi_0^2+\xi_2^2-\xi_0\xi_2-\xi_1^2-\xi_3^2+\xi_1\xi_3)\neq0,\\
&\zeta_2\zeta_3-\zeta_0\zeta_1=\frac{1}{4}(\xi_2\xi_3-\xi_0\xi_1)\neq0,\\
&\zeta_0^2+3\zeta_3^2\neq0.
\end{split}
\end{equation*}
Furthermore, we write
\begin{equation*}
\begin{split}
&\gamma_0=\frac{\sqrt[3]{4}}{6}\zeta_0^{\frac{1}{3}}(\zeta_0\zeta_2+3\zeta_1\zeta_3)^{-\frac{1}{3}}(\zeta_2\zeta_3-\zeta_0\zeta_1)^{-\frac{1}{3}}(\zeta_0^2+3\zeta_3^2)^{-\frac{1}{3}},\\
&\gamma_1=-\gamma_0(3(\zeta_0\zeta_2+3\zeta_1\zeta_3)),\\
&\gamma_2=-\gamma_0(3(\zeta_2\zeta_3-\zeta_0\zeta_1)),\\
&\gamma_3=-\gamma_0(\zeta_0^2+3\zeta_3^2).
\end{split}
\end{equation*}
Suppose $0<\delta<1$ is a sufficiently small constant. Fix
\begin{equation*}
\big|y_j-\gamma_jB^\frac{1}{3}\big|< \delta B^{\frac{1}{3}},
\end{equation*}
for $j=1,2,3$. Then we obtain
\begin{equation*}
\Big|\frac{H_i(\mathbf{y})}{B}-\zeta_i\Big|<\delta f_i(\boldsymbol{\xi}),
\end{equation*}
where each of $f_i( \boldsymbol{\xi})$ is some function in $\boldsymbol{\xi}$. Then there exists $\delta$ sufficiently small, such that
\begin{equation*}
\delta f_i( \boldsymbol{\xi})<\frac{1}{4}\varepsilon.
\end{equation*}
Fix
\begin{equation*}
\mathcal{R}^{(0)}=\{\mathbf{y}\in\mathbb{R}^3:\ y_j\in \big((\gamma_j-\delta),(\gamma_j+\delta)\big),\,\text{for}\ j=1,2,3\}.
\end{equation*}
Then for $\mathbf{y}\in B^{\frac{1}{3}}\mathcal{R}^{(0)}$, we see that (2.16) holds. By Lemma 4, we get
\begin{equation*}
\#\{\mathbf{y}\in B^{\frac{1}{3}}\mathcal{R}^{(0)}\cap\Psi:\,F(\mathbf{y})=P_{20}\}\gg B(\log B)^{-4}.
\end{equation*}
Recall that for $\mathbf{y}\in \Psi$, we have $(F_i(\mathbf{y}),D)=1$. Consequently, we obtain that $(F_0(\mathbf{y}),\dots,F_3(\mathbf{y}))\in\mathbb{Z}^4_{\text{prim}}$. Then the proof is concluded.
\section{The Cayley cubic surface}
\subsection{The circle method}
Suppose that
\begin{equation*}
F(t_0,t_1,t_2,t_3)=\beta_0t_0+\beta_1t_1+\beta_2t_2+\beta_3t_3
\end{equation*}
is a linear form, with the coefficients $\beta_j\in \{-1,1\}$.
Suppose $\eta_0$, $\dots$, $\eta_3$ are fixed positive real numbers, satisfying
\begin{equation*}
F(\eta_0,\eta_1,\eta_2,\eta_3)=\beta_0\eta_0+\beta_1\eta_1+\beta_2\eta_2+\beta_3\eta_3=0.
\end{equation*}
Write
\begin{equation*}
I_j=[\eta_j B^{\frac{1}{3}}-B^{\frac{1}{3}}(\log B)^{-1},\eta_jB^{\frac{1}{3}}+B^{\frac{1}{3}}(\log B)^{-1}],
\end{equation*}
for $j=0,\dots,3$, where $B$ is a sufficiently large parameter. Furthermore, set
\begin{equation*}
R(B)=\sum_{\stackrel{p_j\in I_j}{F(p_0,p_1,p_2,p_3)=0}}(\log p_0)(\log p_1)(\log p_2)(\log p_3).
\end{equation*}
In this section, we use the circle method to prove the following lemma.
\begin{lem}
For any $A>0$, we have
\begin{equation*}
R(B)=J(B)\mathfrak{S}+O(B(\log B)^{-A}),
\end{equation*}
where $J(B)$ is the number of solutions of
\begin{equation*}
F(m_0,m_1,m_2,m_3)=0
\end{equation*}
with $m_j\in I_j $ and
\begin{equation*}
\mathfrak{S}=\prod_{p}\Big(1+\frac{1}{(p-1)^3}\Big).
\end{equation*}
Moreover, we have $J(B)\gg B(\log B)^{-3}$ and $\mathfrak{S}\gg1$.
\end{lem}
\begin{proof}
Write
\begin{equation*}
L=\log B,\qquad P=L^D, \qquad Q=B^{\frac{1}{3}}P^{-3},
\end{equation*}
where $D$ is a sufficiently large parameter to be chosen later. Furthermore, denote
\begin{equation*}
S_j(\alpha)=\sum_{p_j\in I_j}\log p_j e(\beta_jp_j\alpha ).
\end{equation*}
Then we have
\begin{equation*}
R(B)=\int_{\frac{1}{Q}}^{1+\frac{1}{Q}}S_0(\alpha)S_1(\alpha)S_2(\alpha)S_3(\alpha)d \alpha.
\end{equation*}
By Dirichlet's lemma on rational approximation, each $\alpha\in (\frac{1}{Q},1+\frac{1}{Q}]$ may be written in the form
\begin{equation*}
\alpha=\frac{a}{q}+\lambda,\quad |\lambda|<\frac{1}{qQ},
\end{equation*}
for some integers $a$, $q$ with $1\leq a\leq q\leq Q$ and $(a,q)=1$. Now we define the sets of major and minor arcs as follows:
\begin{equation*}
\mathfrak{M}=\bigcup_{q\leq P}\bigcup_{\stackrel{1\leq a\leq q}{(a,q)=1}}\Big[\frac{a}{q}-\frac{1}{qQ},\frac{a}{q}+\frac{1}{qQ}\Big], \quad \mathfrak{m}=\big(\frac{1}{Q},1+\frac{1}{Q}\big]\setminus \mathfrak{M}.
\end{equation*}
Then
\begin{equation*}
R(B)=\int_{\mathfrak{M}}S_0(\alpha)S_1(\alpha)S_2(\alpha)S_3(\alpha)d \alpha+\int_{\mathfrak{m}}S_0(\alpha)S_1(\alpha)S_2(\alpha)S_3(\alpha)d \alpha.
\end{equation*}
We first estimate the contribution of the integral over the major arcs. For any $\alpha\in\mathfrak{M}$, there exist integers $a$ and $q$ such that
\begin{equation*}
\alpha=\frac{a}{q}+\lambda,\quad 1\leq q \leq P,\quad (a,q)=1 \quad \text{and}\quad |\lambda|<\frac{1}{qQ}.
\end{equation*}
Hence
\begin{equation*}
\begin{split}
S_j(\alpha)&=\sum_{p_j\in I_j}\log p_j e\Big(\frac{\beta_jap_j}{q}\Big)e(\beta_j\lambda p_j)\\
&=\frac{1}{\varphi(q)}\sum_{\chi\  \text{mod}\  q}\sum_{\stackrel{h=1}{(h,q)=1}}^{q}e\Big(\frac{\beta_jah}{q}\Big)\bar{\chi}(h)\sum_{p_j\in I_j}\log p_j\chi(p_j)e(\beta_j\lambda p_j).
\end{split}
\end{equation*}
Let
\begin{equation*}
W_j(\chi,\lambda)=\sum_{m_j\in I_j}(\Lambda(m_j)\chi(m_j)-\delta_{\chi})e(\beta_j\lambda m_j),
\end{equation*}
and
\begin{equation*}
\widehat{W}_j(\chi,\lambda)=\sum_{p_j\in I_j}\log p_j \chi(p_j)e(\beta_j\lambda p_j)-\sum_{m_j\in I_j}\delta_{\chi}e(\beta_j\lambda m_j),
\end{equation*}
where $\delta_{\chi}=1$ or $0$ according as $\chi$ is principal or not. Then we have
\begin{equation}
W_j(\chi,\lambda)-\widehat{W}_j(\chi,\lambda)=\sum_{k\geq2}\sum_{p_j^k\in I_j}\log p_j \chi(p_j^k)e(\beta_j\lambda p_j^k)\ll B^{\frac{1}{6}}L.
\end{equation}
Thus
\begin{equation}
\begin{split}
&S_j(\alpha)-\frac{\mu(q)}{\varphi(q)}\sum_{m_j\in I_j}e(\beta_j\lambda m_j)\\
=&\frac{1}{\varphi(q)}\sum_{\chi\,\text{mod}\,q}\sum_{\stackrel{h=1}{(h,q)=1}}^q e\Big(\frac{\beta_jah}{q}\Big)\bar{\chi}(h)\widehat{W}_j(\chi,\lambda)\\
=&\frac{1}{\varphi(q)}\sum_{\chi\,\text{mod}\,q}\sum_{\stackrel{h=1}{(h,q)=1}}^q e\Big(\frac{\beta_jah}{q}\Big)\bar{\chi}(h)(\widehat{W}_j(\chi,\lambda)-W_j(\chi,\lambda))\\
&+\frac{1}{\varphi(q)}\sum_{\chi\,\text{mod}\,q}\sum_{\stackrel{h=1}{(h,q)=1}}^q e\Big(\frac{\beta_jah}{q}\Big)\bar{\chi}(h)W_j(\chi,\lambda).
\end{split}
\end{equation}
The main tool here is a short intervals version of Siegel-Walfisz theorem. Using (6) in Perelli and Pintz \cite{P}, we see that
\begin{equation*}
\sum_{m_j\in I_j}\Lambda(m_j)\chi(m_j)=\sum_{m_j\in I_j}\delta_{\chi}+O(B^{\frac{1}{3}}L^{-1-5D}).
\end{equation*}
Therefore, integration by parts gives
\begin{equation*}
\begin{split}
W_j(\chi,\lambda)=&\int_{I_j}e(\beta_j\lambda u)d\big( \sum_{m_j\leq u,m_j\in I_j}(\Lambda(m_j)\chi(m_j)-\delta_{\chi})\big)\\
\ll& \big|\sum_{m_j\in I_j}(\Lambda(m_j)\chi(m_j)-\delta_{\chi})\big|\\
&+\Big|\lambda\int_{I_j}e(\beta_j\lambda u)\big( \sum_{m_j\leq u,m_j\in I_j}(\Lambda(m_j)\chi(m_j)-\delta_{\chi})\big) du\Big|\\
\ll& (1+|\lambda|B^{\frac{1}{3}}L^{-1})B^{\frac{1}{3}}L^{-1-5D}.
\end{split}
\end{equation*}
Thus we obtain
\begin{equation}
W_j(\chi,\lambda)\ll (1+|\lambda|B^{\frac{1}{3}}L^{-1})B^{\frac{1}{3}}L^{-1-5D}\ll B^{\frac{1}{3}}L^{-2-2D}.
\end{equation}
Combining (3.1), (3.2) and (3.3), we get
\begin{equation*}
S_j(\alpha)=\frac{\mu(q)}{\varphi(q)}\sum_{m_j\in I_j}e(\beta_j\lambda m_j)+O(B^{\frac{1}{3}}L^{-2D}).
\end{equation*}
We have similar results for $S_1(\alpha)$, $S_2(\alpha)$ and $S_3(\alpha)$. Thus we obtain
\begin{equation}
\begin{split}
&\int_{\mathfrak{M}}S_0(\alpha)S_1(\alpha)S_2(\alpha)S_3(\alpha)d \alpha \\
&-\sum_{q\leq P}\frac{{\mu}^2(q)}{{\varphi}^4(q)}\sum_{\stackrel{a=1}{(a,q)=1}}^{q}\int_{-\frac{1}{qQ}}^{\frac{1}{qQ}}\sum_{m_j\in I_j}e(\lambda F(m_0,m_1,m_2,m_3))d\lambda\\
&\ll BL^{-D}.
\end{split}
\end{equation}
Note that
\begin{equation*}
\begin{split}
&\int_{-\frac{1}{2}}^{\frac{1}{2}}\sum_{m_j\in I_j}e(\lambda F(m_0,m_1,m_2,m_3))d\lambda\\
&-\int_{-\frac{1}{qQ}}^{\frac{1}{qQ}}\sum_{m_j\in I_j}e(\lambda F(m_0,m_1,m_2,m_3))d\lambda\ll BL^{-D}.
\end{split}
\end{equation*}
Inserting this into (3.4), we obtain
\begin{equation}
\int_{\mathfrak{M}}S_0(\alpha)S_1(\alpha)S_2(\alpha)S_3(\alpha)d \alpha=J(B)\mathfrak{S}(P)+O(BL^{-D}),
\end{equation}
where
\begin{equation*}
\mathfrak{S}(P)=\sum_{q=1}^{P}\frac{{\mu(q)}^2}{{\varphi(q)}^3}.
\end{equation*}
Now we estimate the contribution of the integral over the minor arcs. For any $\alpha\in\mathfrak{m}$, there exist integers $a$ and $q$ such that
\begin{equation*}
P\leq q\leq Q,\quad (a,q)=1 \quad \text{and} \quad \Big| \alpha-\frac{a}{q}\Big|<\frac{1}{qQ}.
\end{equation*}
Using the non-trivial upper bound for the exponential sums over primes in short intervals, we get
\begin{equation*}
S_0(\alpha)\ll B^{\frac{1}{3}}(\log B)^{-A},
\end{equation*}
provided $D$ is chosen to be sufficiently large. Such a result can be found in several references, for example see Theorem 2 in Zhan \cite{Z}. Also, we have the following mean-value estimate:
\begin{equation*}
\int_{\frac{1}{Q}}^{1+\frac{1}{Q}}|S_j(\alpha)|^2 d\alpha \ll B^{\frac{1}{3}}.
\end{equation*}
By Cauchy's inequality, we obtain
\begin{equation*}
\int_{\frac{1}{Q}}^{1+\frac{1}{Q}}|S_2(\alpha)S_3(\alpha)| d\alpha \ll B^{\frac{1}{3}}.
\end{equation*}
Thus we have
\begin{equation}
\int_{\mathfrak{m}}S_0(\alpha)S_1(\alpha)S_2(\alpha)S_3(\alpha)d \alpha\ll BL^{-A}.
\end{equation}
Combining (3.5) and (3.6), we get
\begin{equation*}
R(B)=J(B)\mathfrak{S}(P)+O(BL^{-A}).
\end{equation*}
For the singular series, we have
\begin{equation*}
\Big|\frac{{\mu(q)}^2}{{\varphi(q)}^3}\Big|\leq \frac{1}{{\varphi(q)}^3}.
\end{equation*}
Note that
\begin{equation*}
\mathfrak{S}=\prod_{p}\Big(1+\frac{1}{(p-1)^3}\Big)=\sum_{q=1}^{\infty}\frac{{\mu(q)}^2}{{\varphi(q)}^3}\gg 1.
\end{equation*}
Thus we obtain
\begin{equation*}
\mathfrak{S}(P)-\mathfrak{S}\ll\sum_{q>P}\frac{1}{{\varphi(q)}^3}\ll L^{-D}.
\end{equation*}
Since $J(B)\ll B(\log B)^{-3}$, then we have
\begin{equation*}
R(B)=J(B)\mathfrak{S}+O(B(\log B)^{-A}).
\end{equation*}
Now we establish the lower bound for $J(B)$. For $j=1,2,3$, we fix $m_j\in[\eta_j B^{\frac{1}{3}}-\frac{1}{3}B^{\frac{1}{3}}(\log B)^{-1},\eta_jB^{\frac{1}{3}}+\frac{1}{3}B^{\frac{1}{3}}(\log B)^{-1}]$. Recall that
\begin{equation*}
\beta_0\eta_0+\beta_1\eta_1+\beta_2\eta_2+\beta_3\eta_3=0,
\end{equation*}
thus we have
\begin{equation*}
\beta_1 m_1+\beta_2 m_2+\beta_3 m_3\in[-\beta_0\eta_0 B^{\frac{1}{3}}-B^{\frac{1}{3}}(\log B)^{-1},-\beta_0\eta_0B^{\frac{1}{3}}+B^{\frac{1}{3}}(\log B)^{-1}].
\end{equation*}
Since
\begin{equation*}
I_0=[\eta_0 B^{\frac{1}{3}}-B^{\frac{1}{3}}(\log B)^{-1},\eta_0B^{\frac{1}{3}}+B^{\frac{1}{3}}(\log B)^{-1}],
\end{equation*}
then we obtain
\begin{equation*}
J(B)\gg B(\log B)^{-3}.
\end{equation*}
Therefore we prove the lemma.
\end{proof}
\subsection{The universal torsor}
In this section, we use a passage to the universal torsor for the Cayley cubic surface. Details can be found in Section 2 of Heath-Brown \cite{Hb}. It is shown that the universal torsor for the Cayley cubic surface is an open subvariety in
\begin{equation*}
\mathbb{A}^{13}=\text{Spec}\mathbb{Z}[v_{01},v_{02},v_{03},y_0,y_1,y_2,y_3,z_{01},z_{02},z_{03},z_{12},z_{13},z_{23}],
\end{equation*}
defined by six equations of the form
\begin{equation*}
z_{ik}z_{il}y_j+z_{jk}z_{jl}y_i=z_{ij}v_{ij}
\end{equation*}
and three equations of the form
\begin{equation*}
v_{ij}v_{ik}=z_{il}^2y_jy_k-z_{jk}^2y_iy_l,
\end{equation*}
where
\begin{equation*}
z_{ij}=z_{ji},\quad v_{ij}=v_{ji},\quad and \quad v_{ij}=-v_{kl}.
\end{equation*}
In fact, we are not working with the full universal torsor. We need the following lemma, which is essentially a restatement of Lemma 1 in Heath-Brown \cite{Hb}.
\begin{lem}
Let $[\mathbf{x}]\in U_2$ be a primitive integral solution. Then either $\mathbf{x}$ or $-\mathbf{x}$ takes the form
\begin{equation*}
\begin{split}
&x_0=z_{01}z_{02}z_{03}y_1y_2y_3,\\
&x_1=z_{01}z_{12}z_{13}y_0y_3y_4,\\
&x_2=z_{02}z_{12}z_{23}y_1y_3y_4,\\
&x_3=z_{03}z_{13}z_{23}y_0y_1y_2,
\end{split}
\end{equation*}
with non-zero integer variables $y_i$ and positive integer variables $z_{ij}$ constrained by the conditions
\begin{equation*}
\text{gcd}(y_i,y_j)=1,
\end{equation*}
\begin{equation*}
\text{gcd}(y_i,z_{ij})=1,
\end{equation*}
\begin{equation*}
\text{gcd}(z_{ab},z_{cd})=1,
\end{equation*}
for $\{a,b\}$, $\{c,d\}$ distinct, and satisfying the equation
\begin{equation*}
z_{12}z_{13}z_{23}y_0+z_{02}z_{03}z_{23}y_1+z_{01}z_{03}z_{13}y_2+z_{01}z_{02}z_{12}y_3=0.
\end{equation*}
Moreover, none of $z_{12}z_{13}z_{23}y_0+z_{02}z_{03}z_{23}y_1$, $z_{12}z_{13}z_{23}y_0+z_{01}z_{03}z_{13}y_2$ or $z_{12}z_{13}z_{23}y_0+z_{01}z_{02}z_{12}y_3$ may vanish. Conversely, if $y_i$ and $z_{ij}$ are as above, then the vector $\mathbf{x}$ taking the above form, will be a primitive integral solution lying on $U_2$.
\end{lem}
\subsection{The proof of the Theorem 2}
Recall that $[\boldsymbol{\xi}]\in U_2$. Thus we get $\xi_0\xi_1\xi_2\xi_3\neq0$, and $\xi_k+\xi_l\neq0$ for $k\neq l$. Now fix
\begin{equation*}
\eta_j=\Big|\frac{\sqrt[3]{\xi_1\xi_2\xi_3\xi_4}}{\xi_j}\Big|,\quad \beta_j=\text{sgn}\Big(\frac{\sqrt[3]{\xi_1\xi_2\xi_3\xi_4}}{\xi_j}\Big),
\end{equation*}
for $j=0,\dots,3$. Then we obtain that $\eta_0$, $\eta_1$, $\eta_2$ and $\eta_3$ are positive numbers, satisfying
\begin{equation*}
\beta_k\eta_k+\beta_l\eta_l\neq0,
\end{equation*}
for $k\neq l$, and
\begin{equation*}
\beta_0\eta_0+\beta_1\eta_1+\beta_2\eta_2+\beta_3\eta_3=0.
\end{equation*}
By Lemma 6, we obtain that for sufficiently large $B$, there exists a suitable positive constant $c$, such that there are at least $cB(\log B)^{-7}$ solutions to the equation
\begin{equation*}
\beta_0p_0+\beta_1p_1+\beta_2p_2+\beta_3p_3=0,
\end{equation*}
with $p_j\in I_j$. Among these solutions, there are at most $O(B^{\frac{2}{3}}(\log B)^{-2})$ ones with $\mu(p_0p_1p_2p_3)=0$. Thus we still have at least $c'B(\log B)^{-7}$ solutions satisfying $\mu(p_0p_1p_2p_3)\neq0$, where $c'$ is a positive constant. In Lemma 7, we fix $z_{01}=z_{02}=z_{03}=z_{12}=z_{13}=z_{23}=1$. Then the relations become
\begin{equation*}
y_0+y_1+y_2+y_3=0,
\end{equation*}
\begin{equation*}
\text{gcd}(y_i,y_j)=1,
\end{equation*}
and none if $y_0+y_1$, $y_0+y_2$ or $y_1+y_3$ may vanish. We also have
\begin{equation*}
\begin{split}
&x_0=y_1y_2y_3,\\
&x_1=y_0y_2y_3,\\
&x_2=y_0y_1y_3,\\
&x_3=y_0y_1y_2.
\end{split}
\end{equation*}
Now we set $y_j= \beta_jp_j$. Thus we obtain
\begin{equation*}
x_i=\prod\limits_{\stackrel{j=0}{j\neq i}}^{3}\beta_jp_j.
\end{equation*}
Now we get $\mathbf{x}\in \mathbb{Z}^4_{\text{prim}}$, $[\mathbf{x}]\in U_1$ and $\mathbf{x}$ primitive. For $x_0$, we have
\begin{equation*}
\Big|\frac{x_0}{B}-\beta_1\beta_2\beta_3\eta_1\eta_2\eta_3\Big|\ll (\log B)^{-1}.
\end{equation*}
Note that $\beta_1\beta_2\beta_3\eta_1\eta_2\eta_3=\xi_0$, then
\begin{equation*}
\Big|\frac{x_0}{B}-\xi_0\Big|\ll (\log B)^{-1}.
\end{equation*}
Arguing similarly as above, we get
\begin{equation*}
\Big|\frac{\mathbf{x}}{B}-\boldsymbol{\xi}\Big|\ll (\log B)^{-1}.
\end{equation*}
Thus the proof of Theorem 2 is concluded.
\bibliographystyle{plain}
\bibliography{saturation}

\end{document}